\documentclass[11pt]{amsart}
\usepackage{amsthm}
\usepackage{amsmath}
\usepackage{amssymb}
\usepackage{enumerate}

\newtheorem{theorem}{Theorem}[section]
\newtheorem{lemma}[theorem]{Lemma}

\newtheorem{definition}[theorem]{Definition}

\newtheorem{thm}{Theorem}

\newcommand \op {{\bf{O}}_{p}}
\newcommand \opd {{\bf{O}}_{p'}}
\newcommand \cent {{\bf{C}}}

\newcommand \bpig {{\textup B}_{\pi}(G)}

\newcommand \ipi {{\textup I}_{\pi}}
\newcommand \ipig {{\textup I}_{\pi}(G)}

\newcommand \ipid {{\textup I}_{\pi'}}
\newcommand \bpid {{\textup B}_{\pi'}}

\newcommand \irr {\textup{Irr}}
\newcommand \irrg {\textup{Irr}(G)}
\newcommand \ibr {{\textup{IBr}}_p}
\newcommand \ibrg {{\textup{IBr}}_p(G)}

\newcommand \ngp {{\bf{N}}_G(P)}
\newcommand \nmp {{\bf{N}}_M(P)}

\newcommand \ngq {{\bf{N}}_G(Q)}

\newcommand \rdz {{\textup{rdz}}}

\newcommand \norm {{\bf{N}}}

\newcommand \dvd {\hbox {\big|}}
\newcommand \ndvd {\hbox {/}\kern-5pt\dvd}
\newcommand \nrml {\lhd}
\def \< {\langle}
\def \> {\rangle}

\begin{document}

\title{Lifts and vertex pairs in solvable groups}


\author{James P. Cossey}
\author{Mark L. Lewis}

\address{Department of Theoretical and Applied Mathematics,
University of Akron,
Akron, OH 44325}

\email{cossey@uakron.edu}

\address{Department of Mathematical Sciences, Kent State
University, Kent, OH 44242}

\email{lewis@math.kent.edu}


\begin{abstract}  Suppose $G$ is a $p$-solvable group, where $p$ is odd.  We explore the connection between lifts of Brauer characters of $G$ and certain local objects in $G$, called vertex pairs.  We show that if $\chi$ is a lift, then the vertex pairs of $\chi$ form a single conjugacy class.  We use this to prove a sufficient condition for a given pair to be a vertex pair of a lift and to study the behavior of lifts with respect to normal subgroups.

\end{abstract}

\maketitle

\section{Introduction}

Throughout this paper $G$ is a finite group and $p$ is a fixed prime.  If $G$ is $p$-solvable and $\varphi \in \ibrg$, the Fong-Swan theorem shows that there is an irreducible character $\chi \in \irrg$ such that $\chi^o = \varphi$, where $\chi^o$ denotes the restriction of $\chi$ to the $p$-regular elements of $G$.  In this case the character $\chi$ is called a {\it lift} of $\varphi$, and in general, if $\chi^o$ is irreducible, then $\chi$ is called {\it a lift}.  It is certainly the case that $\varphi$ could have many lifts, and one active area of research is to understand the set of lifts of $\varphi$ (see \cite{bounds}, \cite{abvertex}, \cite{newnav}).
  
Suppose that $M$ is a normal subgroup of $G$ and $\chi$ is a lift of $\varphi$.  One would like to know under what conditions are the constituents of $\chi_M$ lifts?  By Clifford's theorem, if one constituent of $\chi_M$ is a lift, then they all are.  Note that it is certainly not the case that the constituents of $\chi_M$ must be lifts (see the example at the end of \cite{rmjm}), though if $p$ is odd and $G/M$ is a $p$-group, Navarro has shown that the constituents of $\chi_M$ must be lifts \cite{newnav}.  In this note, we find another sufficient condition that will imply that the constituents of $\chi_M$ are lifts.

Our condition makes use of vertex pairs.  To each ordinary irreducible character $\chi$ of a $p$-solvable group, one can associate a vertex pair $(Q, \delta)$ to $\chi$, where $Q$ is a $p$-subgroup of $G$ and $\delta \in \irr(Q)$ (see Section \ref{pairs} for the precise definition of a vertex pair).  These vertex pairs generalize the vertex subgroups developed by Green \cite{green} and share many of their properties.  In fact, if $p$ is odd and $G$ is $p$-solvable, and $\chi \in \irrg$ has vertex pair $(Q, \delta)$ and is a lift of $\varphi \in \ibrg$, then it is known that $Q$ is a vertex for the irreducible module corresponding to $\varphi$, and $\delta$ is linear \cite{newnav}.  These vertex pairs are \lq \lq local\rq\rq \ objects that yield information on the lifts of the Brauer characters.

With this definition, we can state our condition:

\begin{thm}\label{first}
Let $G$ be $p$-solvable where $p$ is odd, and suppose $M \nrml G$.  Let $\chi \in \irr(G)$ be a lift of $\varphi \in \ibrg$ with vertex pair $(Q, \delta)$, and write $P = Q \cap M$ and $\lambda = \delta_{P}$.  If $P$ is abelian and $\lambda$ is invariant in $\nmp$, then the constituents of $\chi_M$ are lifts.
\end{thm}

In order to prove this theorem, we need to understand the connection between lifts and their vertex pairs.  In particular, we need a uniqueness result regarding vertex pairs that generalizes the main result of \cite{vertodd}.  Given an irreducible character $\chi$ of a $p$-solvable group, there are many different ways to associate a vertex pair to $\chi$, and the resulting vertex pairs need not be conjugate \cite{counterexample}.  We show that if $p$ is odd and $\chi$ is a lift of a Brauer character $\varphi$, then in fact all of the vertex pairs for $\chi$ are conjugate, and thus it makes sense to speak of \lq \lq the\rq\rq \ vertex pair of the lift $\chi$.

\begin{thm}\label{thmA}  Let $p$ be an odd prime and $G$ a $p$-solvable group.  Suppose that $\chi \in \irrg$ is a lift of $\varphi \in \ibrg$.  Then all of the vertex pairs for $\chi$ are conjugate.
\end{thm}

This theorem is known to be false if $p = 2$ (see \cite{counterexample}).  It is also known that the conclusion need not hold if $\chi$ is not a lift (see \cite{vertodd}).

As mentioned before, it is known that if $p$ is odd and $G$ is $p$-solvable, then the vertex subgroup $Q$ of any lift $\chi$ of $\varphi \in \ibrg$ is also the vertex subgroup of the irreducible module corresponding to $\varphi$.  It is not known, however, which characters of $Q$ can be vertex characters of lifts of $\varphi$.  Our last main result gives a sufficient condition for a character $\delta$ of $Q$ to be the vertex character of a lift of $\varphi$.

\begin{thm}\label{thmB} Let $p$ be an odd prime and $G$ a $p$-solvable group.  Suppose $\varphi \in \ibrg$ has vertex subgroup $Q$, and let $\delta \in \irr(Q)$.  If $Q$ is abelian and $\delta$ is invariant in $\ngq$, then there is a unique lift of $\varphi$ with vertex pair $(Q, \delta)$.
\end{thm}

It is not yet known whether the hypotheses that $Q$ is abelian can be weakened.


\section{Restriction to normal subgroups}

In this section, we prove Theorem \ref{first} (assuming Theorems \ref{thmA} and \ref{thmB}).  Before we can prove Theorem \ref{first} we need an easy lemma.  We omit the proof of this lemma, as the first part is Theorem 3.2 of \cite{vertnormal} (and can also be found in \cite{laradjinormal}), and the proof of the second part consists of the exact same argument used to prove Theorem 1.1 of \cite{rmjm} (which used the version of Theorem \ref{thmA} when $|G|$ is assumed to be odd - note that we now know we need only assume that $G$ is $p$-solvable and $p$ is odd).

\begin{lemma}\label{normlemma}  Let $p$ be a prime and let $G$ be a $p$-solvable group, and let $\varphi \in \ibrg$ have vertex subgroup $Q$.  Suppose $M \nrml G$.  Then:

\begin{itemize}

\item[(a)]  There is some constituent $\theta$ of $\varphi_M$ that has vertex $Q \cap M$.

\item[(b)]  If $p$ is odd, and $\chi$ is a lift of $\varphi$ with vertex pair $(Q, \delta)$, then some constituent of $\chi_M$ has vertex pair $(Q \cap M, \delta_{Q \cap M})$.

\end{itemize}
\end{lemma}

We will need to make use of a result of Navarro \cite{actions} regarding relative defect zero characters.

\begin{definition}  Let $p$ be a prime.  If $G$ is a group with $N \nrml G$, and $\mu \in \irr(N)$ and $\chi \in \irr(G \mid \mu)$, then we say that $\chi$ has relative defect zero (with respect to $p$) if $$(\chi(1)/\mu(1))_p = |G : N|_p.$$
\end{definition}

If the prime $p$ is clear from the context, we will simply refer to the relative defect zero characters over $\mu$.  We will denote the relative defect zero characters of $G$ lying over $\mu$ by ${\textup{rdz}}(G \mid \mu)$.  Note that ${\textup{rdz}}(G \mid 1_N)$ consists of precisely the defect zero characters of $G/N$.

The following, which is a restatement of Theorem 2.1 of \cite{actions}, will be key to our arguments.  Note there are no conditions on the group $G$ or the prime $p$ (other than of course $|G|$ is finite).

\begin{theorem}\label{rdzlemma}  Let $G$ be a finite group and $p$ a prime.  Let $D \nrml G$ be a $p$-subgroup and let $\mu \in \irr(N)$ be $G$-invariant.  Then there is a bijection $\chi \rightarrow \chi_{\mu}$ from the defect zero characters of $G/D$ to $\rdz(G \mid \mu)$.  If $\mu$ is linear, then $\chi^o = \chi_{\mu}^o$.
\end{theorem}

We have not taken the time to generalize Theorem \ref{rdzlemma} from Brauer characters to $\pi$-partial characters.  It is for this reason that the results in this section and Section \ref{suff} are stated in terms of Brauer characters.  However, the results in this section and Section \ref{suff} are true for $\pi$-partial characters with $2 \in \pi$.

We will need to make use of character triple isomorphisms.  The definition and key results for character triple isomorphisms can be found in Definition 11.23 of \cite{text} and the results in Chapter 11 after that.  This next lemma gives a connection between lifts and character triple isomorphisms.   
The argument for this next lemma appeared in the proof of Corollary B of \cite{abvertex}.

\begin{lemma}\label{char trip}
Let $G$ be a $p$-solvable group, let $N$ be a normal $p'$-subgroup of $G$, and let $\theta \in \irr (N)$ be $G$-invariant.  Then $(G,N,\theta)$ is character triple isomorphic to $(G^*,N^*,\theta^*)$ where $N^*$ is a central, $p'$-subgroup of $G^*$.  Let $\chi \in \irr(G \mid \theta)$ correspond to $\chi^* \in \irr(G^* \mid \theta^*)$.  Then $\chi$ is a lift if and only if $\chi^*$ is a lift.  Furthermore, suppose $\psi$ is a lift of $\varphi \in \ibrg$, and let $\varphi^* = (\psi^*)^o$.  Then the number of lifts of $\varphi$ is equal to the number of lifts of $\varphi^*$.  
\end{lemma}

\begin{proof}
By Theorem 5.2 of \cite{pipart}, there is a character triple
$(G^*,N^*,\theta^*)$ which is isomorphic to $(G,N,\theta)$ and where $N^*$ is a central, $p'$-subgroup.  Take $H$ to be a Hall $p$-complement of $G$.  Let $H^*$ correspond to $H$, and note that $H^*$ is a Hall $p$-complement of $G^*$.  By the Fong-Swan theorem, $\chi^o$ is not irreducible if and only if there there exist characters $\alpha, \beta$ such that $\chi^o = \alpha^o + \beta^o$.  This occurs if and only if $\chi_H = \alpha_H + \beta_H$.  Using the character triple isomorphism, this
is equivalent to $(\chi^*)_{H^*} = (\alpha^*)_{H^*} +
(\beta^*)_{H^*}$ and to $(\chi^*)^o = (\alpha^*)^o + (\beta^*)^o$.
We conclude that $\chi^o$ is irreducible if and only if $(\chi^*)^o$
is irreducible. 

Suppose $\psi$ is a lift of $\varphi$, then we define
$\varphi^* = (\psi^*)^o$. Notice that $\chi \in \irr(G)$ is a lift of $\varphi$ if and only if $\chi_H = \varphi_H$.   It follows that $\chi$ is a lift of $\varphi$ if and only if $\chi^*$ is a lift of $\varphi^*$, and so, the number of lifts of $\varphi$ equals the number of lifts of $\varphi^*$.
\end{proof}

We will also need to understand how vertex pairs behave with respect to the character triple isomorphism.  (We will need the basic results about $p$-special, $p'$-special, and $p$-factorable characters, see \cite{bpi} for example.) 

\begin{lemma} \label{char trip pair}  Let $G$ be a $p$-solvable group and let $N$ be a normal $p'$-subgroup of $G$, and suppose $\theta \in \irr(N)$ is $G$-invariant.  Let $(G^*, N^*, \theta^*)$ be an isomorphic character triple where $N^*$ is a central $p'$-subgroup of $G^*$.  If $\chi \in \irr(G | \theta)$ is a lift with vertex pair $(Q, \delta)$, then there exists a subgroup $Q^* \cong Q$ of $G^*$ and a character $\delta^* \in \irr(Q^*)$ such that $(Q^*, \delta^*)$ is a vertex pair for $\chi^*$.  Moreover, $\delta$ is invariant in $\ngq$ if and only if $\delta^*$ is invariant in $\norm_{G^*}(Q^*)$.
\end{lemma} 

\begin{proof}  Suppose that $\chi \in \irrg$ is a lift with vertex pair $(Q, \delta)$.  Then there is a subgroup $U$ containing $QN$ and a factorable character $\alpha \beta$ of $U$ (where $\alpha$ is $p'$-special and $\beta$ is $p'$-special) that induces $\chi$, and $Q$ is a Sylow $p$-subgroup of $U$ and $\beta_Q = \delta$.  Notice that $\alpha$ lies over $\theta$ and $N$ is in the kernel of $\beta$, so $\alpha \beta \in \irr(U | \theta)$.  Thus $(\alpha \beta)^* \in \irr(U^* | \theta^*)$ is a factorable character that induces $\chi^*$.  Write $(\alpha \beta)^* = \alpha_1 \beta_1$ (where $\alpha_1$ is $p'$-special and $\beta_1$ is $p$-special), let $Q^*$ be a Sylow $p$-subgroup of $U^*$, and write $\delta^* = (\beta_1)_{Q^*}$.  Then $Q \cong Q^*$, and $(Q^*, \delta^*)$ is a vertex pair for $\chi^*$.

To complete the proof, we show that with the above notation, $\delta$ is invariant in $\ngq$ if and only if $\delta^*$ is invariant in $\norm_{G^*}(Q^*)$.  Notice that $\delta$ has a unique $p$-special extension $\epsilon \in \irr(QN)$, and that $\delta$ is invariant in $\ngq$ if and only if $\epsilon$ is invariant in $\norm_G(QN)$.  Let $\widehat{\theta}$ denote the unique $p'$-special extension of $\theta$ to $QN$, and note that $\widehat{\theta} \epsilon \in \irr(QN | \theta)$.  Moreover, $\alpha \beta \in \irr(U)$ lies over $\widehat{\theta} \epsilon$.  Now $\delta$ is invariant in $\ngq$ if and only if $\widehat{\theta} \epsilon$ is invariant in $\norm_G(QN)$, which occurs if and only if $(\widehat{\theta} \epsilon)^*$ is invariant in $\norm_{G^*}((QN)^*)$ (by the properties of a character triple isomorphism), which occurs if and only if the $p$-special factor $\epsilon_1$ of $(\widehat{\theta} \epsilon)^*$ is invariant in $\norm_{G^*}((QN)^*)$.  Note that $\epsilon_1$ necessarily restricts to a vertex character for $\chi^*$ obtained from $(U^*, (\alpha \beta)^*)$, so $(\epsilon_1)_{Q^*} = \delta^*$.  Finally, note that $\epsilon_1$ is invariant in $\norm_{G^*}((QN)^*)$ if and only if $\delta^*$ is invariant in $\norm_{G^*}(Q^*)$, and we have proven the lemma.

\end{proof}

We now turn to Theorem \ref{first}, and in fact, we prove more. Note that in the statement of the following theorem, we do not assume that the vertex subgroup $Q$ is abelian, only that a relevant subgroup of $Q$ is abelian.  Also, since $p$ is odd, then the vertex character for the lift $\chi$ is necessarily linear.

\begin{theorem}\label{normal}  Let $G$ be $p$-solvable where $p$ is odd, and suppose $M \nrml G$.  Let $\chi \in \irr(G)$ be a lift of $\varphi \in \ibrg$ with vertex pair $(Q, \delta)$, and write $P = Q \cap M$ and $\lambda = \delta_{P}$.  If $P$ is abelian and $\lambda$ is invariant in $\nmp$, then the constituents of $\chi_M$ are lifts.  Moreover, if $\lambda$ is invariant in $\ngp$ and $\psi$ is a constituent of $\chi_M$, then $G_{\psi} = G_{\psi^o}$.
\end{theorem}

\begin{proof}  
To prove the first statement, it is enough to show that some constituent of $\chi_M$ is a lift.  There is some constituent $\psi$ of $\chi_M$ such that the Clifford correspondent $\rho$ of $\chi$ in $\irr(G_{\psi} \mid \psi)$ has vertex pair $(Q, \delta)$.  Notice that by part (b) of Lemma \ref{normlemma}, $\psi$ has vertex pair $(P, \lambda)$.  If $G_{\psi}$ is proper in $G$, then by induction, we see that $\psi$ is a lift.

Thus, we may assume that $\psi$ is invariant in $G$, and note that this implies, by part (b) of Lemma \ref{normlemma}, that $\psi$ has vertex $(P, \lambda)$.  Let $N = \opd(M)$ and note that $N \nrml G$.  Let $\alpha$ be a constituent of $\chi_N$, and thus also of $\psi_N$.  Let $T = G_{\alpha}$ and assume that $T < G$.  By replacing $\alpha$ by a conjugate if necessary, we may assume that the Clifford correspondent $\mu \in \irr(T \mid \alpha)$ of $\chi$ has vertex pair $(Q, \delta)$.  Note that a Frattini argument (since $\psi$ is invariant in $G$) shows that $G = M G_{\alpha}$.  By induction, we see that the unique constituent $\nu$ of $\mu_{N_{\alpha}}$ is a lift, and necessarily lies over $\alpha$.  Since $\nu^o \in \ibr(N_{\alpha} \mid \alpha)$ induces irreducibly to $N$, then $\psi = \nu^N$ is a lift.

Therefore we may assume that $\alpha$ is invariant in $G$.  We may now use Lemma \ref{char trip} and Lemma \ref{char trip pair} to replace the triple $(G, N, \alpha)$ with an isomorphic character triple without losing the information about the lifts or their vertex pairs.  Thus, we may assume that $N$ is a central $p'$-subgroup of $G$ (and hence, also central in $M$).  Also, note that by part (a) of Lemma \ref{normlemma}, some constituent $\theta$ of $\varphi_M$ has vertex subgroup $P$.

Let $K \supseteq N$ be such that $K/N = \op(M/N)$.  Since $N$ is a central $p'$-subgroup in $M$, then $K = N \times S$, where $S = \op(M)$.  Since $P$ is a vertex subgroup of $\theta$, then $S \subseteq P$.  Also, since $P$ is abelian, then $P \subseteq \cent_M(S)$, so $PN/N \subseteq \cent_{M/N}(SN/N) \subseteq SN/N$, where the last containment is by the Hall-Higman lemma.  Therefore, $P \subseteq S$, and thus $P = S \nrml M$ and by assumption, $\lambda$ is invariant in $M$.  By Theorem \ref{rdzlemma}, since $\theta$ is a Brauer character of $M$ with vertex $P$, there is a unique character in $\rdz(M \mid \lambda)$ that lifts $\theta$.  However, we know that $\psi$ has vertex $(P, \lambda)$, and thus $\psi \in \rdz(M \mid \lambda)$, and therefore $\psi$ is a lift of $\theta$.

To prove the second statement, notice that we have shown that $\psi$ has vertex pair $(P, \lambda)$, and thus by Theorem \ref{thmB}, $\psi$ is the unique lift of $\psi^o$ with vertex pair $(P, \lambda)$.  It is clear that $G_{\psi} \subseteq G_{\psi^o}$.  To prove the reverse containment, we may without loss of generality assume that $\psi^o$ is invariant in $G$, and prove that $\psi$ is invariant in $G$.  Note that by a Frattini argument, $G = M \ngp$.  Since we are assuming $\ngp$ stabilizes $\lambda$, then $G = M \norm_G(P, \lambda)$.  Let $g \in G$, and write $g = mn$, where $m \in M$ and $n \in \norm_G(P, \lambda)$.  Then $\psi^g = \psi^n$.  But $\psi^n$ is a lift of $(\psi^o)^n = \psi^o$ and has vertex pair $(P, \lambda)^n = (P, \lambda)$, and thus, by the uniqueness in Theorem \ref{thmB}, we see that $\psi^n = \psi$, and therefore $\psi^g = \psi$ and $\psi$ is invariant in $G$.
\end{proof}

\section{Theorem \ref{thmB}} \label{suff}

In this section we prove Theorem \ref{thmB} (using Theorem \ref{thmA}).  We note that the proof of Theorem \ref{thmB} bears many similarities with the proof of Theorem \ref{normal}.
  
Before beginning the proof, we show that it can certainly be the case that there are irreducible characters of a vertex subgroup $Q$ that are not vertex characters of lifts.  Let $p=7$ and let $Q$ have order $7$, and let $T$ have order $3$ and act nontrivially on $Q$, and let $G = Q \rtimes T$ be the semidirect product.  Let $\varphi \in \ibr(G/Q)$ be nontrivial, and note that $\varphi$ must be linear and $\varphi$ has vertex subgroup $Q$.  If $\delta \in \irr(Q)$ is nontrivial, then any character of $G$ lying over $\delta$ must have degree $3$ and thus cannot be a lift of $\varphi$.  Thus there are no lifts of $\varphi$ with vertex pair $(Q, \delta)$.  This example shows that we cannot remove the hypothesis that $\delta$ is invariant in $\ngq$ from Theorem \ref{thmB}.

We now prove Theorem \ref{thmB} of the introduction.

\begin{proof} [Proof of Theorem \ref{thmB}]
Let $N = \opd(G)$, and let $\alpha \in \irr(N)$ be a constituent of $\varphi_N$.  Take $T$ to be the stabilizer of $\alpha$, and suppose that $T < G$.  Replacing $\alpha$ by a conjugate if necessary, we may assume that $Q$ is a vertex subgroup for the Clifford correspondent $\eta \in \ibr(T \mid \alpha)$ of $\varphi$.  Clearly, the hypotheses are inherited by $\eta$ in $\ibr(T)$, and thus, by induction, there is a unique lift $\psi \in \irr(T)$ of $\eta$ with vertex pair $(Q, \delta)$.  Note that $\psi^G \in \irrg$, and $\varphi = \eta^G = (\psi^o)^G = (\psi^G)^o$, and thus $\psi^G$ is a lift of $\varphi$.  By Theorem \ref{thmA}, any vertex pair for $\psi$ is a vertex pair for $\psi^G$, and thus $(Q, \delta)$ is a vertex pair of $\psi^G$.

We still need to show that $\psi^G$ is the unique such lift of $\varphi$.  Suppose $\chi_1$ and $\chi_2$ are lifts of $\varphi$ with vertex pair $(Q, \delta)$.  Again, choose $\alpha \in \ibr(N)$ so that the Clifford correspondent $\eta$ of $\varphi$ has vertex subgroup $Q$.  Now the Clifford correspondents $\psi_1$ and $\psi_2$ of $\chi_1$ and $\chi_2$ have vertex pairs $(Q, \delta_1)$ and $(Q, \delta_2)$, respectively.  In light of Theorem \ref{thmA}, $\delta_1$ and $\delta_2$ are conjugate to $\delta$ via elements in $\ngq$.  By assumption, $\ngq$ stabilizes $\delta$, and thus $\delta_1 = \delta_2 = \delta$.  Therefore, $\psi_1$ and $\psi_2$ are lifts of $\eta$ with vertex pair $(Q, \delta)$, and thus, by induction, $\psi_1 = \psi_2$.  We conclude that $\chi_1 = \chi_2$.

Therefore, we may assume that $\alpha$ is invariant in $G$.  Using Lemma \ref{char trip} and Lemma \ref{char trip pair}, we may replace $(G,N,\alpha)$ using a character triple isomorphism without losing information about the lifts or their vertex pairs, and thus, we may assume that $N$ is a central $p'$-subgroup of $G$.  Let $K \supseteq N$ be such that $K/N = \op(G/N)$.  Since $N$ is a central $p'$-subgroup in $G$, then $K = N \times P$, where $P = \op(G)$.  Since $Q$ is a vertex subgroup of $\varphi$, then $P \subseteq Q$.  Also, since $Q$ is abelian, then $Q \subseteq \cent_G(P)$, so $QN/N \subseteq \cent_{G/N}(PN/N) \subseteq PN/N$, where the last containment is by the Hall-Higman lemma.  Therefore, $Q \subseteq P$, and thus $Q = P \nrml G$ and therefore by assumption, $\delta$ is invariant in $G$.

Now, $\varphi \in \ibrg$ has a normal vertex subgroup $Q$, and we may view $\varphi$ as a Brauer character of $G/Q$ which has defect zero.  Thus, the unique character $\chi \in \irr(G/Q)$ that lifts $\varphi$ has defect zero.  Applying Theorem \ref{rdzlemma}, there is a unique character $\chi_{\delta} \in \rdz(G \mid \delta)$ that lifts $\varphi$.  Since any lift of $\varphi$ that has vertex $(Q, \delta)$ must lie above $\delta$ and have relative defect zero, we are done.
\end{proof}

\section{Generalized vertices} \label{pairs}

In this section we will prove Theorem \ref{thmA}.  Rather than work with Brauer characters, in this section we work in the context of Isaacs' partial characters to prove a slightly more general result.  Hence, we will have a set of primes $\pi$.  To define the $\pi$-partial characters, one needs to assume that $G$ is $\pi$-separable.  As in the context of Brauer characters, we let $G^o$ denote the set of $\pi$-elements in $G$.  Given an ordinary character $\chi$, we use $\chi^o$ to denote the restriction of $\chi$ to $G^o$.  The $\pi$-partial characters of $G$ are the functions defined on $G^o$ that are restrictions of ordinary characters.  The $\pi$-partial characters that cannot be written as the sum of two other partial characters are called irreducible.  We use $\ipig$ to denote the irreducible $\pi$-partial characters of $G$.  For a full exposition on $\pi$-partial characters, we refer the reader to \cite{bpi} and \cite{pipart}.

The irreducible $\pi$-partial characters of $G$ have many properties in common with the irreducible Brauer characters of a $p$-solvable group.  (In fact, if $\pi = p'$, then $\ipig = \ibrg$, and the requirement that $p$ is odd is equivalent to $2 \in \pi$.)  For example, we can define induction of partial characters from subgroups in the same way one defines induction of Brauer characters.  Given an irreducible $\pi$-partial character $\varphi$ of $G$, we can define a vertex $Q$ for $\varphi$ to be a Hall $\pi'$-subgroup of a subgroup $U$ that contains a $\pi$-partial character $\kappa$ of $\pi$-degree that induces $\varphi$.  Isaacs and Navarro proved in \cite{IsNa} that all of the vertices for $\varphi$ are conjugate in $G$.  (A different proof of this fact is in \cite{istan}.)  There also exists a Clifford correspondence for $\pi$-partial characters.  If $G$ is $\pi$-separable and $N \nrml G$ and $\theta \in \ipi(N)$, then induction is a bijection from the set $\ipi(G_{\theta} \mid \theta)$ to $\ipi(G \mid \theta)$ (see \cite{Fong}).

We also need to consider $\pi$-special characters.  Let $G$ be a $\pi$-separable group.  A character $\chi \in \irrg$ is $\pi$-special if $\chi (1)$ is a $\pi$-number and for every subnormal group $M$ of $G$, each irreducible constituent of $\chi_M$ has determinantal order that is a $\pi$-number.  Many of the basic results of $\pi$-special characters can be found in Section 40 of \cite{hupte} and Chapter VI of \cite{Mawo}.  One result that is proved is that if $\alpha$ is $\pi$-special and $\beta$ is $\pi'$-special, then $\alpha \beta$ is necessarily irreducible. Furthermore, if $\alpha'$ is $\pi$-special and $\beta'$ is $\pi$-special so that $\alpha' \beta' = \alpha \beta$, then $\alpha' = \alpha$ and $\beta' = \beta$.  We say that $\chi$ is $\pi$-factored (or factored, if the $\pi$ is clear from context) if $\chi = \alpha \beta$, where $\alpha$ is $\pi$-special and $\beta$ is $\pi'$-special.  Another result is that if $H$ is a Hall $\pi$-subgroup of $G$, then restriction defines an injection from the $\pi$-special characters of $G$ into $\irr (H)$.

Following the terminology introduced in \cite{vertodd}, we say $(Q,\delta)$ is a {\it generalized $\pi$-vertex} for $\chi \in \irrg$ if there exists a pair $(U,\psi)$ (where $U \subseteq G$ and $\psi \in \irr(U)$) so that $\psi^G = \chi$, $Q$ is a Hall $\pi$-complement of $U$, $\psi = \alpha \beta$ where $\alpha$ is $\pi$-special and $\beta$ is $\pi'$-special, and $\beta_Q = \delta$.  In this context, we say that $(U,\psi)$ is a {\it generalized $\pi$-nucleus} for $\chi$.

In \cite{vertodd}, the first author proved that if $|G|$ is odd and $\chi \in \irrg$ is such that $\chi^o \in \ipig$, then the generalized $\pi$-vertices for $\chi$ are conjugate. We now show that the hypothesis that $|G|$ is odd can be replaced by the hypothesis that $G$ is $\pi$-separable and $2 \in \pi$.  Our argument will parallel the argument in \cite{vertodd}.

The main result, which is the $\pi$-version of Theorem \ref{thmA} is the following.

\begin{theorem}\label{second main}
Let $\pi$ be a set of primes with $2 \in \pi$, and let $G$ be a $\pi$-separable group. If $\chi \in \irrg$ is such that $\chi^o \in \ipig$, then all of the generalized $\pi$-vertices for $\chi$ are conjugate.
\end{theorem}

The key to our work is a recent result of Navarro.  Replacing $p$ by a set of primes $\pi$ with $2 \in \pi$, the proof of Lemma 2.1 of \cite{newnav} proves:

\begin{lemma}\label{navarro}
Let $\pi$ be a set of primes with $2 \in \pi$, and let $G$ be a $\pi$-separable group.  Let $\chi \in \irrg$ be $\pi'$-special.  If $\chi (1) > 1$, then $\chi^o$ is not in $\ipig$.
\end{lemma}

For the remainder of this section, our work will parallel the work in \cite{vertodd}.  The following should be compared with Lemma 2.3 of \cite{vertodd}.

\begin{lemma} \label{lemma 2.3}
Let $\pi$ be a set of primes with $2 \in \pi$, and let $G$ be a $\pi$-separable group.  Let $\chi \in \irrg$ be such that $\chi^o \in \ipig$.  If $U \le G$ and $\psi \in \irr(U)$ is a $\pi$-factored character that induces $\chi$, then the $\pi'$-special factor of $\psi$ is linear.  Moreover, if $Q$ is a Hall $\pi$-complement of $U$, then $Q$ is a vertex subgroup of $\chi^o$.
\end{lemma}

\begin{proof}
Note that since $\chi^o \in \ipig$, and $\psi^G = \chi$, then
$\psi^o \in \ipi(U)$.  Since $\psi$ is $\pi$-factored, we have $\psi = \alpha \beta$ where $\alpha$ is $\pi$-special and $\beta$ is $\pi'$-special.  It follows that $\beta^o \in \ipi(U)$.  By Lemma \ref{navarro}, $\beta (1) = 1$.  It follows that $\psi$ has $\pi$-degree and $\psi^o \in \ipi(U)$.  By Theorem B of \cite{IsNa}, $Q$ is a vertex subgroup of $\chi^o$.
\end{proof}

The next lemma is similar to Lemma 3.1 of \cite{vertodd}.

\begin{lemma} \label{lemma 3.1}
Let $\pi$ be a set of primes with $2 \in \pi$, and let $G$ be a $\pi$-separable group.  Let $\chi \in \irrg$ be such that $\chi = \alpha \beta$ where $\alpha$ is $\pi$-special and $\beta$ is linear and $\pi'$-special.  Suppose $\psi \in \irr(U)$ is $\pi$-factored and induces $\chi$.  If $\delta$ is the $\pi'$-special factor of $\psi$, then $\beta_U = \delta$.
\end{lemma}

\begin{proof}
Note that  $\alpha = \alpha \beta \beta^{-1} = \chi \beta^{-1}$.  It
follows that $(\psi \beta^{-1}|_U)^G = \psi^G \beta^{-1} = \chi \beta^{-1} = \alpha$.  Since $\alpha$ is $\pi$-special, we may use Theorem C of \cite{induct} to see that $\psi \beta^{-1}|_U$ is $\pi$-special.  We can write $\psi = \gamma \delta$ where $\gamma$ is $\pi$-special.  Now, $\gamma^o = \psi^o = (\psi
\beta^{-1}|_U)^o$, and so, $\gamma = \psi \beta^{-1}|_U = \gamma \delta \beta^{-1}|_U$. It follows that $\delta \beta^{-1}|_U = 1_U$, and hence, $\delta = \beta_U$.
\end{proof}

The next result should be compared with Lemma 3.2 of \cite{vertodd}.  Let $\pi$ be a set of primes with $2 \in \pi$ and suppose $G$ is $\pi$-separable.  We will need the basic properties of the set $\bpig \subseteq \irrg$ introduced in \cite{bpi}.  In particular, we need to know that restriction to $G^o$ gives a bijection from $\bpig$ to $\ipig$ and that the $\pi$-special characters of $G$ are precisely the characters of $\pi$-degree in $\bpig$.  We will also use the magic field automorphism that was described in \cite{odd}.  We write $\sigma$ to denote the magic field automorphism.  Let $\chi \in \irrg$. In \cite{odd}, it is proved that $\chi \in \bpig$ if and only if $\chi^\sigma = \chi$ and $\chi^o \in \ipig$.

\begin{lemma} \label{lemma 3.2}
Let $\pi$ be a set of primes with $2 \in \pi$, and let $G$ be a $\pi$-separable group with subgroup $U$.  Suppose $\chi \in \irrg$ satisfies $\chi^o \in \ipig$.  Assume $\psi \in \irr(U)$ is $\pi$-factored so that $\chi = \psi^G$.  Suppose $|G:U|$ is a $\pi$-number and the $\pi'$-special factor of $\psi$ extends to $G$. Then $\chi$ is $\pi$-factored.
\end{lemma}


We will use the notation $\beta'$ to denote the restriction of an ordinary  character $\beta$ of $G$ to the $\pi'$-elements of $G$.

\begin{proof}
Let $\psi = \alpha \beta$ where $\alpha$ is $\pi$-special and
$\beta$ is $\pi'$-special.  Let $\varphi = \beta' \in \ipid(U)$. Let $\eta \in \irrg$ be an extension of $\beta$.  Now, $(\eta')_U = (\eta_U)' = \beta' \in \ipid(U)$.  It follows that $\eta' \in \ipid(G)$.  Let $\delta \in \bpid(G)$ so that $\delta' = \eta'$. Observe that $\delta (1) = \eta (1) = \beta (1)$ is a $\pi'$-number, and so $\delta$ is $\pi'$-special.  Also, $(\delta_U)' = \beta'$ implies that $\delta_U \in \irr(U)$.  By Theorem A of \cite{induct}, $\delta_U$ is $\pi'$-special.  This implies that $\delta_U = \beta$.

We now have $\chi = \psi^G = (\alpha \beta)^G = \alpha^G \delta$.  This implies that $\alpha^G \in \irrg$.  Notice that $(\alpha^G)^\sigma = (\alpha^\sigma)^G = \alpha^G$.  Also, $\chi^o = (\alpha^G \delta)^o = (\alpha^G)^o \delta^o \in \ipig$, and so, $(\alpha^G)^o \in \ipig$.  It follows that $\alpha^G$ is $\pi$-special.  We conclude that $\chi$ is $\pi$-factored.
\end{proof}

The next result is similar to Corollary 3.3 of \cite{vertodd}.

\begin{lemma} \label{cor 3.3}
Let $\pi$ be a set of primes with $2 \in \pi$, and let $G$ be a $\pi$-separable group.  Let $\chi \in \irrg$ be $\pi$-factored and have $\pi$-degree.  Let $N$ be a normal subgroup of $G$ and suppose $\theta \in \irr(N)$ is a constituent of $\chi_N$.  Let $T$ be the stabilizer of $\theta$ in $G$.  If $\psi \in \irr(T \mid \theta)$ is the Clifford correspondent for $\chi$ with respect to $\theta$, then $\psi$ is $\pi$-factored.
\end{lemma}

\begin{proof}
Observe that $\theta$ is $\pi$-factored.  We can write $\chi =
\gamma \delta$ and $\theta = \alpha \beta$ where $\gamma$ and
$\alpha$ are $\pi$-special and $\delta$ and $\beta$ are
$\pi'$-special.  Since $\chi$ has $\pi$-degree, $\delta (1) = 1$ and thus, $\delta_N = \beta$.  It follows that $T$ is the stabilizer of $\alpha$ in $G$.  We take $\mu \in \irr(T \mid \alpha)$ to be the Clifford correspondent for $\gamma$ with respect to $\alpha$. We have $\gamma (1) = |G:T| \mu (1)$, and thus, $\mu (1)$ is a $\pi$-number. Observe that $(\mu^o)^G = (\mu^G)^o = \gamma^o \in \ipig$ and thus, $\mu^o \in \ipi(T)$.  Since $\alpha^\sigma = \alpha$, we have $\mu^\sigma \in \irr(T \mid \alpha)$.  Since $(\mu^\sigma)^G = (\mu^G)^\sigma = \mu^G$, it follows that $\mu^\sigma = \mu$, and we conclude that $\mu$ is $\pi$-special.  Because $\delta$ is linear and $\pi'$-special, $\delta_T$ is $\pi'$-special.  We see that $(\mu\delta_T)^G = \mu^G \delta = \gamma \delta= \chi$.  Also, $(\mu\delta_T)_N = \mu_N \delta_N$, and so, $\alpha \beta = \theta$ is a constituent of $(\mu\delta_T)_N$.  We obtain $\mu\delta_T \in \irr(T \mid \theta)$.  Since $(\mu\delta_T)^G = \chi = \psi^G$, we can use the Clifford correspondence to see that $\psi = \mu \delta_T$.  Therefore, $\psi$ is $\pi$-factored.
\end{proof}

Since the proof of the next lemma is essentially the proof of Lemma 3.4 of \cite{vertodd} where Lemma \ref{lemma 2.3} is used in place of Lemma 2.3 of \cite{vertodd}, we do not include it here.

\begin{lemma} \label{lemma 3.4}
Let $\pi$ be a set of primes with $2 \in \pi$, and let $G$ be a $\pi$-separable group.  Let $\chi \in \irrg$ be a lift of $\varphi \in \ipig$, and suppose $N$ is normal in $G$ such that the constituents of $\chi_N$ are $\pi$-factored.  Suppose $\psi \in \irr(U)$ is $\pi$-factored, and suppose $\psi^G = \chi$.  Then $|NU:U|$ is a $\pi$-number.
\end{lemma}


This next lemma is similar to Lemma 3.5 of \cite{vertodd}.

\begin{lemma} \label{lemma 3.5}
Let $\pi$ be a set of primes with $2 \in \pi$, and let $G$ be a $\pi$-separable group.  Let $\chi \in \irrg$ be a lift of $\varphi \in \ipig$.  Suppose $\psi$ is a $\pi$-factored character of some subgroup $H$ of $G$ that induces $\chi$, and suppose there is a normal subgroup $N$ of $G$ such that the constituents of $\chi_N$ are $\pi$-factored and $G = NH$.  Then $\chi$ is $\pi$-factored and the $\pi'$-special factor of $\chi$ restricts irreducibly to the $\pi'$-special factor of $\psi$.
\end{lemma}

\begin{proof}
Notice that the second conclusion follows from the first conclusion by Lemma \ref{lemma 3.1}.  We assume the first conclusion is not true, and we take $G$, $N$, and $H$ to be a counterexample with $|G:H| + |N|$ minimal.

By Lemma \ref{lemma 2.3}, the $\pi'$-special factor of $\psi$ is linear, so $\psi (1)$ is a $\pi$-number.  Applying Lemma \ref{lemma 3.4}, we see that $|G:H| = |HN:N|$ is a $\pi$-number.  Since $\chi (1) = |G:H| \psi (1)$, we see that $\chi$ has $\pi$-degree.

Choose $K$ normal in $G$ so that $N/K$ is a chief factor for $G$.  Notice that the irreducible constituents of $\chi_K$ are $\pi$-factored.  If $G = HK$, then $G$, $K$, and $H$ form a counterexample with $|G:H| + |K| < |G:H| + |N|$ violating the choice of minimal counterexample.  Thus, we have $HK < G$.

Notice that $G = NH = N(HK)$.  Notice that $\psi^{HK} \in \irr(HK)$ will be a lift of a partial character in $\ipi(HK)$.  Also, the irreducible constituents of $(\psi^{HK})_K$ are constituents of $\chi_K$, and thus must be factored.  If $H < HK$, then $|HK:H| + |K| < |G:H| + |N|$, and so $HK$, $K$, and $H$ cannot form a counterexample.  Thus, $\psi^{HK}$ must be factored and induce $\chi$.  Also, $|G:HK| + |N| < |G:H| + |N|$, so $G$, $HK$, and $N$ do not form a counterexample.  We conclude that $\chi$ is $\pi$-factored, a contradiction.  This implies that $H = HK$.

We have $K \le H$.  Let $\eta$ be an irreducible constituent of $\psi_K$.  Notice that $\eta^N$ has an irreducible constituent $\theta$ which is a constituent of $\chi_N$, so $\theta$ and $\eta$ are both $\pi$-factored.  Since $\chi$ has $\pi$-degree, $\theta$ has a linear $\pi'$-special factor.  If $\nu$ is the $\pi'$-special factor of $\eta$, then $\nu$ extends to both the $\pi'$-special factor of $\theta$ and the $\pi'$-special factor of $\psi$.  This implies that $\nu$ is invariant in both $N$ and $H$.  Since $G = NH$, we conclude that $\nu$ is $G$-invariant.

Note that $|N : K|$ divides the $\pi$-number $|G: H|$ and thus $|N : K|$ is a $\pi$-number.  Let $\hat\nu$ be the unique $\pi'$-special extension of $\nu$ to $N$, and since $\nu$ is $G$-invariant so is $\hat\nu$.  We can now apply Corollary 4.2 of \cite{bpi} to see that restriction defines a bijection from $\irr(G \mid \hat\nu)$ to $\irr(H \mid \hat\nu_{N \cap H})$.  Observe that the $\pi'$-special factor of $\psi$ will belong to $\irr(H \mid \hat\nu_{N \cap H})$ since $\hat\nu_{N \cap H}$ is the unique $\pi'$-special extension of $\nu$ to $N \cap H$.  It follows that the $\pi'$-special factor of $\psi$ extends to $G$, and applying Lemma \ref{lemma 3.2} we conclude that $\chi$ is factored, as desired.
\end{proof}

We make use of the normal nucleus constructed by Navarro in \cite{npi}.  We quickly summarize this construction.  Fix a character $\chi \in \irrg$.  Navarro shows that there is a unique subgroup $N$ that is maximal subject to being normal in $G$ and the irreducible constituents of $\chi_N$ are $\pi$-factored.  If $N = G$, then take $(G,\chi)$ to be the normal nucleus of $\chi$.  If $N < G$, let $\theta$ be an irreducible constituent of $\chi_N$.  Navarro shows that in this case $\theta$ is not $G$-invariant.  We then let $\chi_{\theta} \in \irr(G_{\theta} \mid \theta)$ be the Clifford correspondent for $\chi$ with respect to $\theta$.  We define the normal nucleus for $\chi$ to be the normal nucleus of $\chi_\theta$ which can be computed inductively since $G_{\theta} < G$.  Note that the process terminates when we have a factorable character, and thus the normal nucleus character of $\chi$ is factorable and induces to $\chi$.  (The definition of the normal nucleus is obviously motivated by Isaacs' construction of the subnormal nucleus in \cite{bpi}.)  It can be easily seen that all of the normal nuclei for $\chi$ are conjugate.

The proof of Theorem \ref{second main} is essentially the proof of Theorem 4.1 of \cite{vertodd}, and thus we do not include it here in full detail.  However, we do provide a brief sketch of the proof.  The goal is to show that if $(U, \psi)$ is any generalized $\pi$-nucleus of $\chi$, then the generalized $\pi$-vertex of $\chi$ defined by $(U, \psi)$ is conjugate to a vertex for $\chi$ arising from a normal nucleus.  Lemmas \ref{lemma 2.3} and \ref{lemma 3.1} allow us to assume that $\chi$ is not factorable.  Let $N \nrml G$ be maximal so that the constituents of $\chi_N$ are factorable.  By Lemma \ref{lemma 3.4}, we see that $|NU : U|$ is a $\pi$-number, and thus Lemma \ref{lemma 3.5} allows us to replace the pair $(U, \psi)$ with $(NU, \psi^{NU})$, and thus we may assume $N \subseteq U$.  Letting $\theta$ be a constituent of $\psi_N$, we use Lemma \ref{cor 3.3} and Lemma \ref{lemma 3.1} to replace $(U, \psi)$ with the pair $(U_{\theta}, \xi)$, where $\xi$ is the Clifford correspondent for $\psi$ in $\irr(U_{\theta} |\theta)$.  We finish by applying the inductive hypothesis to the group $G_{\theta}$ and the Clifford correspondent for $\chi$ lying over $\theta$, which by definition has a normal nucleus in common with $\chi$.



\end{document}